\documentclass[a4paper, 11pt]{article}

\usepackage[
            includefoot,  
            marginparwidth=0in,     
            marginparsep=0in,       
            margin=1.45in,               
            includemp]{geometry}

\usepackage{bbm}
\usepackage{float}
\usepackage{graphicx}                  
\usepackage{amssymb}
\usepackage{amsfonts}
\RequirePackage{amsmath}
\RequirePackage{amsthm}

\usepackage{color}

\usepackage{fancyhdr}

\newtheorem{thm}{Theorem}[section]
\newtheorem{defn}[thm]{Definition}
\newtheorem{prop}[thm]{Proposition}
\newtheorem{lem}[thm]{Lemma}

\newtheorem{cor}[thm]{Corollary}

\newtheorem{rmq}{Remark}[section]

\DeclareMathOperator{\Ent}{Ent}
\DeclareMathOperator{\id}{Id}
\DeclareMathOperator{\Tr}{Tr}
\DeclareMathOperator{\Id}{Id}

\newcommand{\R}{\mathbb{R}}
\newcommand{\N}{\mathbb{N}}

\begin{document}

\title{Higher-order Stein kernels for Gaussian approximation}
\author{Max Fathi}
\date{\today}

\maketitle

\begin{abstract}
We introduce higher-order Stein kernels relative to the standard Gaussian measure, which generalize the usual Stein kernels by involving higher-order derivatives of test functions. We relate the associated discrepancies to various metrics on the space of probability measures and prove new functional inequalities involving them. As an application, we obtain new explicit improved rates of convergence in the classical multidimensional CLT under higher moment and regularity assumptions.  
\end{abstract}

\section{Introduction}
Stein's method is a set of techniques, originating in works of Stein \cite{Ste72, Ste86}, to bound distances between probability measures. We refer to \cite{Cha14, Ros11} for a recent overview of the field. The purpose of this work is a generalization of one particular way of implementing Stein's method when the target measure is Gaussian, which is known as the Stein kernel approach. 

Let $\mu$ be a probability measure on $\R^d$. A matrix-valued function $\tau_{\mu} : \R^d \longrightarrow \mathcal{M}_d(\R)$ is said to be a \emph{Stein kernel} for $\mu$ (with respect to the standard Gaussian measure $\gamma$ on $\R^d$) if for any smooth test function $\varphi$ taking values in $\mathbb{R}^d$,  we have
\begin{equation} \label{eq_stein}
\int{x \cdot \varphi d\mu} = \int{\langle \tau_{\mu}, \nabla \varphi \rangle_{\mathrm{HS}} d\mu}. 
\end{equation}
For   applications, it is generally enough to consider the restricted class of test functions $\varphi$ satisfying $\int (|\varphi|^2 + \|\nabla \varphi\|_{\mathrm{HS}}^2 )d\mu <\infty$, in which case both integrals in \eqref{eq_stein} are well-defined as soon as  $\tau_{\mu} \in L^2(\mu)$, provided $\mu$ has finite second moments.  

The  motivation behind the definition is that, since the standard centered Gaussian measure $\gamma$ is the only probability distribution on $\R^d$ satisfying the integration by parts formula
\begin{equation} \label{ibp_gauss}
\int{x \cdot \varphi d\gamma} = \int{\operatorname{div} (\varphi)d\gamma},
\end{equation} 
the Stein kernel $\tau_{\mu}$ coincides with the identity matrix, denoted by $\id$, if and only if the measure $\mu$ is  equal to $\gamma$. Hence,  a Stein kernel can be used to control how far $\mu$ is from being a standard Gaussian measure in terms of how much it violates the integration by parts formula \eqref{ibp_gauss}. This notion appears implicitly in many works on Stein's method, and has recently been the topic of more direct investigations \cite{AMV10, Cha09, NPS14a, LNP15, Fat18}. 

However, \eqref{ibp_gauss} is not the only integration by parts formula that characterizes the Gaussian measure. For example, in dimension one, the standard Gaussian measure is characterized by the relation
$$\int{H_{k}(x)f(x)d\gamma(x)} = \int{H_{k-1}(x)f'(x)d\gamma(x)}$$
for all smooth test functions $f$, where the $H_k$ are the Hermite polynomials $H_k(x) = (-1)^k e^{x^2/2}\frac{d^k}{dx^k}e^{-x^2/2}$. The case $k = 1$ corresponds to the standard formula \eqref{ibp_gauss}. While these are not the only integration by parts formulas one could state, they are in some sense the most natural ones, due to the role Hermite polynomials play as eigenfunctions of the Ornstein-Uhlenbeck generator. 

Before defining higher-order Stein kernels, we must define a few notations. $\mathcal{T}_{k,d}$ shall denote the space of $k$-tensors on $\R^d$, that is $k$-dimensional arrays of size $d$, and $\mathcal{T}_{k,d}^{sym}$ the subspace of symmetric tensors, that is arrays $A$ such that for any permutation $\sigma \in \mathcal{S}_k$ and $i_1,...,i_k \in \{1,...,d\}$ we have $A_{i_1...i_k} = A_{i_\sigma(1)...\sigma(k)}$. In particular, differentials of order $k$ of smooth functions belong to $\mathcal{T}_{k,d}^{sym}$, $\mathcal{T}_{1,d} = \R^d$ and $\mathcal{T}_{2,d} = \mathcal{M}_d(\R)$. We equip these spaces with their natural Euclidean structure and $L^2$ norm, which we shall respectively denote by $\langle \cdot, \cdot \rangle$ and $||\cdot ||_2$. 

In dimension $d \geq 1$, Hermite polynomials are defined as follows: 

\begin{defn}[Multi-dimensional Hermite polynomials] 
For $k \geq 1$ and indices $i_1,..,i_d \in \N$ such that $i_1 + ..+i_d = k$, we define the Hermite polynomial $H_k^{i_1,..,i_d} \in \R[X_1,..,X_d]$ as
$$H_k^{i_1,..,i_d}(x_1,..,x_d) := (-1)^k e^{|x|^2/2} \frac{d^k}{dx_1^{i_1}..dx_d^{i_d}}e^{-|x|^2/2}.$$
We also define $\bar{H}_k := (-1)^ke^{|x|^2/2}d^k(e^{-|x|^2/2})$, so that the coefficients of the $k$-tensor $\bar{H}_k$ are the $H_k^{i_1,..,i_d}$. 
\end{defn}

The natural generalization of the notion of Stein kernels with respect to the integration by parts formulas defined via multidimensional Hermite polynomials would be to say that a $(k+1)$-tensor $\tau_k$ is a $k$-th order Stein kernel for $\mu$ is for any smooth $f : \R^d \longrightarrow \mathcal{T}_{k,d}$ we have
$$\int{\langle \bar{H}_{k}, f \rangle d\mu} = \int{\langle \tau_k, Df \rangle d\mu}.$$
However, it turns out that for the applications we shall describe below, it is more convenient to define higher-order Stein kernels in a different way: 

\begin{defn}[Higher-order Stein kernels]
We define Stein kernels of order $k$ $\bar{\tau}_k$ (as long as they exist) as any symmetric $(k+1)$-tensor satisfying
\begin{equation} \label{eq_def_stein_ho}
\int{\langle \bar{\tau}_k, D^{k}f \rangle d\mu} = \int{x \cdot  f - \Tr(\nabla f) d\mu}
\end{equation}
for all smooth vector-valued $f$ such that $x \cdot  f - \Tr(\nabla f)$ and $||D^{k}f||^2$ are integrable with respect to $\mu$. 
\end{defn}

Note that $\bar{\tau}_1 = \tau - \Id$ where $\tau$ is a classical Stein kernel. The choice of restricting the definition to symmetric tensors is non-standard when $k=1$. It is motivated by the fact that since we only test out the relation on tensors of the form $D^{k+2}f$, which are symmetric, and will allow us to easily relate the expectation of such kernels to moments of the underlying measure. 

In some sense, the point of view we develop here is very close to the one developed in \cite{GR97}, where approximate Stein identities with higher-order derivatives are used, in the framework of the zero-bias transform. The main advantage of the functional-analytic framework presented here is to allow more explicit estimates in the multivariate setting, albeit under strong regularity conditions. A particular upside of our estimates is that the dependence on the dimension will be very explicit. 

A first remark is that we have the iterative relation
\begin{equation} \label{eq_stein_ho_iter}
\int{\langle \bar{\tau}_k, D^{k}f \rangle d\mu} = \int{\langle\bar{\tau}_{k-1}, D^{k-1}f\rangle d\mu}
\end{equation}
As we shall later see in Lemma \ref{lem_moments}, for $\bar{\tau}_k$ to exist, we must have $\int{P(x_1,..,x_d)d\mu} = \int{P(x_1,..,x_d)d\gamma}$ for any $P \in \R_{k+1}[X_1,..,X_d]$. Of course, this is not a sufficient condition. These kernels are in some sense centered, so that $\mu$ is Gaussian iff $\bar{\tau}_k = 0$. For $k=1$, this does not exactly match with the usual definition, which is not centered, but this shift will make notations much lighter. 

These Stein kernels can be related to kernels associated with Hermite polynomials via linear combinations. For example, if $\tau_2$ (resp. $\tau_1$) is a kernel associated with Hermite polynomials of degree 2 (resp. 1), then $\bar{\tau}_2 = \tau_2 - x \otimes \tau_1$ is a second-order Stein kernel in the sense of \eqref{eq_def_stein_ho}. 

As for classical Stein kernels, we can then define the associated discrepancy, which measures how far a given probability measure is from satisfying the associated Gaussian integration by parts formula. 

\begin{defn}
The $k$-th order Stein discrepancy is defined by 
$$\bar{S}_k(\mu) := \inf \int{||\bar{\tau}_k||^2d\mu},$$
where the infimum is over all possible Stein kernels of order $k$ for $\mu$, since they may not be unique. 
\end{defn}

\begin{rmq}
The abstract setting we use here is not restricted to Hermite polynomials or higher-order derivatives. For example, it would be possible to define a kernel by considering \emph{any} tensor-valued function $u: \R^d \longrightarrow \mathcal{T}_{k,d}$ and looking for a function $K_{\mu}(u) : \R^d \longrightarrow \mathcal{T}_{k+1,d}$ such that for any smooth function $f : \R^d \longrightarrow \mathcal{T}_{k,d}$ we would have 
$$\int{\langle u, f\rangle d\mu} = \int{\langle K_{\mu}(u), \nabla f \rangle d\mu}.$$
This more general point of view is related to the one developed in \cite{MRS18}. Existence would be treated in the same way as we shall implement in this work, but we do not have any other example leading to meaningful applications at this point. 
\end{rmq}

The main application of these higher-order Stein kernels to the rate of convergence in the classical CLT is the following decay estimate, made precise in Corollary \ref{cor_w2_rate}: if the random variables $(X_i)$ are iid, isotropic, centered and have mixed moments of order three equal to zero, then if $\mu_n$ is the law of the renormalized sum in the CLT we have an estimate of the form
$$W_2(\mu_n, \gamma) \leq \frac{Cd^{1/2}(1+ \log n)}{n}$$
where $C$ is a constant we shall make precise, that depends on a regularity condition on the law of the $X_i$. This seems to be the first improved rate of convergence in the multidimensional CLT in $W_2$ distance.  

The plan of the sequel is as follows: in Section 2, we shall establish basic properties of higher-order Stein kernels, including existence and some first results on what distances the associated discrepancies control. In Section 3, we shall establish some functional inequalities relating Wasserstein distances, entropy and Fisher information. Finally, in Section 4, we shall derive various improved bounds on the rate of convergence in the central limit theorem under moment constraints. 

\section{Properties}

\subsection{Existence}

Before studying these higher-order Stein kernels and their applications, the first question to ask is when do they actually exist? As for classical Stein kernels, there must be some condition beyond normalizing the moments, since they may not exist for measures with purely atomic support. 

The first condition we can point out is that existence of Stein kernels constrain the values of certain moments: 

\begin{lem} \label{lem_moments}
Assume that $\mu$ admits Stein kernels $\bar{\tau}_j$ up to order $k \geq 1$. Then for any polynomial $P$ in $d$ variables of degree $\ell \leq k$ we have $\int{P(x)d\mu} = \int{P(x)d\gamma}$, and moreover if this is also true for polynomials of degree $k+1$ then 
$$\int{(\bar{\tau}_k)_{i_1,...i_{k+1}}d\mu} = 0$$
for any indices $i_1,...,i_k$. 
\end{lem}

\begin{proof}
We prove this statement by induction on $k$. The case $k=1$ can be readily checked by testing the Stein identity on coordinates $x_1,..,x_d$. Assume the statement holds for $k \geq 1$. To prove the statement for $k+1$, it is enough to check it for monomials of degree $k$, by the induction assumption. Up to relabeling, we can restrict to the case where the degree in $x_1$ is positive. Let $\alpha_1,.., \alpha_d$ such that $\sum \alpha_i = k$ and $P(x) = x_1^{\alpha_1}x_2^{\alpha_2}...x_d^{\alpha_d}$. Define $E = \{(i_1,...,i_k); \forall \ell \hspace{1mm} |\{j; i_j = \ell\} = \alpha_{\ell}\}$. We have
\begin{align*}
\int{x_1^{\alpha_1+1}..x_n^{\alpha_n}d\mu} &= \int{\langle \bar{\tau}_{k,1}, D^k P \rangle d\mu} + \sum_i \int{(\alpha_i x_i^{\alpha_i-1}) \prod_{j \neq i}x_j^{\alpha_j} d\mu} \\
&= \prod_{j = 1}^{d} (\alpha_j !) \sum_{(i_1,...,i_k) \in E} \int{(\bar{\tau}_k)_{1,i_1,...,i_k} d\mu} + \sum_i \int{(\alpha_i x_i^{\alpha_i-1}) \prod_{j \neq i}x_j^{\alpha_j} d\gamma} \\
&= \left(\prod_{j = 1}^{d} (\alpha_j !) \right) k! \int{\bar{\tau}_{1,1..1,2...}d\mu}+ \sum_i \int{(\alpha_i x_i^{\alpha_i-1}) \prod_{j \neq i}x_j^{\alpha_j} d\gamma} \\
\end{align*}
where we have used the symmetry of $\bar{\tau}_k$, and the moment assumption to match the second term. The indices in the last line corresponds to having $\alpha_i$ times the indice $i$, and the order does not matter by symmetry of $\bar{\tau}_k$. Since for a Gaussian measure the two integrals of moments match, the integral of the kernel must be zero as soon as the moment assumption is satisfied. 
\end{proof}

In dimension one, when $\mu$ has a nice density $p$ with respect to the Lebesgue measure, we can give explicit formulas in terms of $p$: 

\begin{prop}
Let $\mu(dx) = p(x)dx$ be a probability measure on $\R$ with connected support, such that $\int{x^jd\mu} = \int{x^j d\gamma}$ for all $j \leq k$. Then 
the iterative formula  
$$\bar{\tau}_k(x) = -\frac{1}{p(x)}\int_{x}^{\infty}{\bar{\tau}_{k-1}(y)p(y)dy}$$ 
defines Stein kernels, with $\bar{\tau}_1 = -\frac{1}{p(x)}\int_{x}^{\infty}{yp(y)dy} - 1$ the usual explicit formula for classical Stein kernels in dimension one. 
\end{prop}

We refer to \cite{Sau18} for a detailed study of 1st order kernels in dimension one. We shall not develop this point of view further, and focus on the situation in higher dimension, where this formula is no longer available. It turns out that, up to extra moment conditions, the arguments used in \cite{CFP17} for standard Stein kernels also apply. Before stating the conditions, we must first define Poincar\'e inequalities: 

\begin{defn}
A probability measure $\mu$ on $\R^d$ satisfies a Poincar\'e inequality with constant $C_P$ if for all locally lipschitz function $f$ with $\int{fd\mu} = 0$ we have
$$\int{f^2d\mu} \leq C_P\int{|\nabla f|^2d\mu}.$$

\end{defn}

Poincar\'e inequalities are a standard family of inequalities in stochastic analysis, with many applications, such as concentration inequalities and rates of convergence to equilibrium for stochastic processes. See \cite{BGL14, BBCG08} and references therein for background information and conditions ensuring such an inequality holds. 

Our basic existence result is the following: 
\begin{thm}
Assume that $\mu$ satisfies a Poincar\'e inequality with constant $C_P$, and that its moments of order less than $k$ match with those of the standard Gaussian. Then a Stein kernel of order $k$ exists, and moreover $\bar{S}_k(\mu)^2 \leq C_P^{k-1}(C_P-1)d$. 
\end{thm}

This theorem yields a sufficient condition for existence, but it is not necessary. Even in the case $k=1$, we do not know of a useful full characterization of the situations where Stein kernels exist. Actually, \cite{CFP17} uses a more general type of functional inequality to ensure existence of a 1st order Stein kernel, but its extension to higher order kernels is a bit cumbersome, since the condition would iteratively require previous kernels to have a finite 2nd moment after multiplication with an extra weight. 

\begin{proof}
We proceed by induction. The case $k=1$ was proven in \cite{CFP17}. Assume that the statement is true for some $k$, and that $\mu$ has moments of order less than $k+1$ matching with those of the Gaussian. Let $\bar{\tau}_k$ be a Stein kernel of order $k$ for $\mu$, which exists by the induction assumption. We wish to prove existence of $\bar{\tau}_{k+1}$. Consider the functional
$$J(f) = \frac{1}{2}\int{|D^{k+1} f|^2d\mu} - \int{\langle \bar{\tau}_k, D^k f \rangle d\mu}$$
defined for $f : \R^d \longrightarrow \R^d$. It is easy to check that, from the Euler-Lagrange equation for $J$, if $g$ is a minimizer of $J$, then $\bar{\tau}_{k+1} := D^{k+1} g$ satisfies \eqref{eq_def_stein_ho}. 

From the Poincar\'e inequality and the fact that $\bar{\tau}_k$ is centered due to the moment assumption, we have
$$\left|\int{\langle \bar{\tau}_k, D^k f \rangle d\mu} \right|^2 \leq C_P\left(\int{|\bar{\tau}_k|^2d\mu}\right)\left(\int{|D^{k+1} f|^2d\mu}\right),$$
so that $f \longrightarrow \int{\langle \bar{\tau}_k, f \rangle d\mu}$ is a continuous linear form w.r.t. the norm $\int{|\nabla f|^2d\mu}$. Hence from the Lax-Milgram theorem (or Riesz representation theorem) we deduce existence (and uniqueness) of a centered global minimizer $g$, and $\bar{\tau}_{k+1} = D^{k+1} g$ is a suitable Stein kernel, and satisfies the symmetry assumption. Moreover, 
\begin{align*}
-\frac{1}{2}&\int{|\bar{\tau}_{k+1}|^2d\mu} = \frac{1}{2}\int{|D^{k+1} g|^2d\mu} - \int{\langle \bar{\tau}_k, g \rangle d\mu} \\
&\geq \frac{1}{2}\int{|\nabla g|^2d\mu}-\frac{C_P}{2}\bar{S}_k(\mu)^2 - \frac{1}{2C_P}\int{\left|g - \int{gd\mu}\right|^2d\mu} \\
&\geq -\frac{C_P}{2}\bar{S}_k(\mu).
\end{align*}
The induction assumption then yields $\bar{S}_{k+1}(\mu)^2 \leq C_P^{k}(C_P-1)d.$
\end{proof}

\subsection{Topology}

In this section, we are interested in studying what distances between a probability measure and a Gaussian are controlled by our discrepancies. As is classical in Stein's method, we seek to control a distance of the form
$$d(\mu, \nu) = \underset{f \in \mathcal{F}}{\sup} \hspace{1mm} \int{fd\mu} - \int{fd\nu}$$
where the class of test functions $\mathcal{F}$ should be symmetric, and large enough to indeed separate probability measures. The total variation distance corresponds to the set of functions bounded by one, while the $L^1$ Kantorovitch-Wasserstein distance is obtained when considering the set of 1-lipschitz functions, thanks to the Kantorovitch-Rubinstein duality formula \cite{Vil03}. 

To relate such distances to Stein's method, we introduce the Poisson equation
\begin{equation} \label{eq_poisson}
\Delta h - x\cdot \nabla h = f - \int{fd\gamma}. 
\end{equation}
The classical implementation is that if the solution satisfies a suitable regularity bound, then we can control 
$$\int{fd\mu} - \int{fd\gamma} = \int{(\Delta h - x\cdot \nabla h)d\mu}$$
by a type of Stein discrepancy. Due to the elliptic nature of the Ornstein-Uhlenbeck generator, the solution $h$ gains some regularity compared to $f$. For example, if $f$ is $1$-lipschitz, $h$ is $C^{3-\epsilon}$ \cite{GMS18}. Here, to control it by a Stein discrepancy of order $k$, we shall have to differentiate several times the solution, and require it to satisfy a bound of the form $||D^{k+1}h||_{\infty} \leq C$. In particular, solutions to the Poisson equation should be smooth enough, which typically requires $f$ to be $C^{k}$ (this will be explained in more details in the proof of Theorem \ref{thm_comp_zol} below). Hence we introduce

We can now state a first result on the topology controlled by higher-order Stein discrepancies.  

\begin{thm} \label{thm_comp_zol}
Let $\mu$ be a probability measure on $\R^d$ whose $k+1$ first mixed moments match with those of a $d$-dimensional standard centered Gaussian. Then
$$d_{Zol,k}(\mu, \gamma) := \underset{||D^{k}f|| \leq 1}{\sup} \hspace{1mm} \int{fd\mu} - \int{fd\gamma} \leq  \bar{S}_k(\mu).$$
\end{thm}

The controlled distance $d_{Zol,k}$ can be thought of as a generalization of the $L^1$ Kantorovitch-Wasserstein distance, which corresponds to $k=1$. It is known as the Zolotarev distance of order $k$, and it controls the same topology as the $L^k$ Kantorovitch-Wasserstein distance \cite{BH00}, that is weak convergence and convergence of moments up to order $k$. 

\begin{proof}
We first derive a regularity estimate for solutions of the Poisson equation. The scheme of proof below is a straightforward extension of the regularity bound of \cite{CM08} in the case $k=1$. Similar regularity bounds, in operator norm, for arbitrary $k$ where derived in \cite{Gau16}. In the case where $f$ is lipschitz, better regularity bounds (namely, $C^{2,1-}$ bounds) were obtained in \cite{GMS18}, and it should be possible to get better regularity bounds for general $k$. However, for our purpose it is not clear that improved bounds would further help us here. 

As pointed out by Barbour \cite{Bar90}, a solution of the Poisson equation \eqref{eq_poisson} is given by
\begin{equation}
h_f(x) = \int_0^1{\frac{1}{2t}\int{(f(\sqrt{t}x + \sqrt{1-t}y) - f(y))d\gamma(y)}dt}. 
\end{equation}
and after integrating by parts with respect to the Gaussian measure, its gradient can be represented as 
$$\nabla h_f(x) = \int_0^1{\frac{1}{2\sqrt{t(1-t)}}\int{yf(\sqrt{t}x + \sqrt{1-t}y)d\gamma(y)}dt}$$
and hence higher-order derivatives are given by
$$D^{k+1}h_f(x) = \int_0^1{\frac{1}{2\sqrt{t(1-t)}}\int{y\otimes D^kf(\sqrt{t}x + \sqrt{1-t}y)d\gamma(y)}dt}.$$
We then have for any $A \in \mathcal{T}_{k+1,d}$ 
\begin{align*}
\langle D^{k+1}h_f(x), A \rangle &= \int_0^1{\frac{t^{k/2}}{2\sqrt{t(1-t)}}\int{\langle A, y\otimes D^kf(\sqrt{t}x + \sqrt{1-t}y)\rangle d\gamma(y)}dt} \\
&= \int_0^1{\frac{t^{k/2}}{2\sqrt{t(1-t)}}\int{\langle Ay,  D^kf(\sqrt{t}x + \sqrt{1-t}y)\rangle d\gamma(y)}dt} \\
&\leq \underset{z}{\sup} ||D^kf(z)||_2 \int_0^1{\frac{1}{2\sqrt{t(1-t)}}dt}\int{||Ay||_2d\gamma(y)} \\
&\leq  \sup_z ||D^kf(z)||_2 \left(\int{\underset{i_1,..,i_k}{\sum} \left(\underset{j}{\sum} A_{i_1,..,i_k,j}y_j\right)^2d\gamma(y)}\right)^{1/2} \\
&= \sup_z ||D^kf(z)||_2||A||_2.
\end{align*}
Therefore 
$$\underset{x}{\sup} ||D^{k+1}h_f(x)||_2 \leq \sup_x||D^kf(x)||_2.$$

We then have, for any function $f$ satisfying $\sup_x ||D^k f(x)||_2 \leq 1$, 
\begin{align*}
\int{fd\mu} -\int{fd\gamma} &= \int{\Delta h_f - x \cdot \nabla h_f d\mu} \\
&= \int{\langle \bar{\tau}_1, D^2 h_f \rangle d\mu} \\
&= \int{\langle \bar{\tau}_k, D^{k+1} h_f \rangle d\mu} \\
&\leq \bar{S}_k(\mu).
\end{align*}
This concludes the proof. 
\end{proof}

\begin{rmq}
In dimension one, the Ornstein-Uhlenbeck enjoys strictly better regularization properties, which would allow to control stronger distances.  
\end{rmq}

\section{Functional inequalities}

Our first functional inequality is a generalization of the HSI inequality of \cite{LNP15}.

\begin{thm}[HSI inequalities] \label{thm_hsi} Let $k \geq 2$. We have
$$H(\mu) \leq \frac{1}{2}\min \left(I(\mu), kI(\mu)^{(k-1)/k}\bar{S}_k(\mu)^{2/k}\right).$$
\end{thm}

This inequality improves on the classical Gaussian logarithmic Sobolev inequality of Gross \cite{Gro75}. 

We introduce the Ornstein-Uhlenbeck semigroup 
$$P_tf (x) = \mathbb{E}[f(e^{-t}x + \sqrt{1 - e^{-2t}}G)]$$
where $G$ is a standard Gaussian random variable. The properties of this semigroup have been well-studied. In particular, as time goes to infinity, $P_tf$ converges to $\int{fd\gamma}$, and the entropy and Fisher information are related by De Brujin's formula: 
\begin{equation} \label{debrujin}
H(f) = \int_0^{\infty}{I(P_tf)dt}.
\end{equation}

The key lemma at the core of our results is the following estimate on Fisher information along the flow: 

\begin{lem} \label{lem_fisher_decay}
For any $t > 0$, we have 
$$I(\mu_t) \leq \frac{e^{-2(k+1)t}}{(1 - e^{-2t})^{k}}k!\bar{S}_k(\mu)^2.$$
\end{lem}

When $k=1$, this estimate corresponds to the main result of \cite{NPS14a}, and played a core role in the proofs of the functional inequalities of \cite{LNP15}. This extension to higher orders will allow us to get more precise estimates when higher-order Stein kernels exist, i.e. under moment constraints. 

\begin{proof}
We have the commutation relation
\begin{equation}
\partial_i(P_tf) = e^{-t} P_t(\partial_i f). 
\end{equation}

Following \cite{LNP15}, we have a representation formula for the Fisher information along the Ornstein-Uhlenbeck flow: 

\begin{align*}
I(\mu_t) &= \frac{e^{-2t}}{\sqrt{1 - e^{-2t}}}\int{\int{\langle (\tau(x) - \operatorname{Id})y, \nabla v_t(e^{-t}x + \sqrt{1-e^{-2t}}y) \rangle d\mu(x)}d\gamma(y)} \\
&= \frac{e^{-(k+1)t}}{\sqrt{1 - e^{-2t}}}\int{\int{\langle \bar{\tau}_k(x)y, D^{k} v_t(e^{-t}x + \sqrt{1-e^{-2t}}y) \rangle d\mu(x)}d\gamma(y)} \\
&= \frac{e^{-(k+1)t}}{(1 - e^{-2t})^{k/2}}\int{\int{\langle \bar{\tau}_k(x)\bar{H}_{k}(y), \nabla v_t(e^{-t}x + \sqrt{1-e^{-2t}}y) \rangle d\mu(x)}d\gamma(y)}
\end{align*}

Applying the Cauchy-Schwarz inequality and integrating out in $y$, we get the result. 
\end{proof}
\begin{proof} [Proof of Theorem \ref{thm_hsi}]
From \eqref{debrujin} and the decay property of the Fisher information $I(\mu_t) \leq e^{-2t}I(\mu)$, we deduce that for any $t \geq 0$ we have
$$H(\mu) \leq \frac{1-e^{-2t}}{2}I(\mu) + \int_t^{\infty}{I(\mu_s)ds}.$$
Using Lemma \ref{lem_fisher_decay} on the second term, we get
\begin{align*}
H(\mu) &\leq \frac{1-e^{-2t}}{2}I(\mu) + k!\bar{S}_k(\mu)^2\int_t^{\infty}{\frac{e^{-2(k+1)s}}{(1 - e^{-2s})^{k}}ds} \\
&\leq \frac{1-e^{-2t}}{2}I(\mu) + \frac{k!}{2(k-1)}\bar{S}_k(\mu)^2 \frac{e^{-2kt}}{(1-e^{-2t})^{k-1}} \\
&\leq \frac{1-e^{-2t}}{2}I(\mu) + \frac{k!\bar{S}_k(\mu)^2}{2(1-e^{-2t})^{k-1}}.
\end{align*}
We optimize by taking $t$ such that $1-e^{-2t} = \left(\frac{k! \bar{S}_k^2}{I}\right)^{1/k}$ if possible, and otherwise $t = \infty$ (which boils down to the usual logarithmic Sobolev inequality), and we get the result. We used the easy bound $(k!)^{1/k} \leq k$ to simplify the expression. 
\end{proof}

We can also obtain functional inequalities controlling the $W_2$ distance. Recall that in the case $k= 1$, \cite{LNP15} established the inequality
$$W_2(\mu, \gamma)^2 \leq \bar{S}_1(\mu)^2,$$
which itself reinforced classical bounds on the $W_1$ distance via Stein's method, and allows to get simple proofs of CLTs in $W_2$ distance, since Stein discrepancies turn out to me easier to estimate in some situations. Our result is the following variant involving higher-order discrepancies: 

\begin{thm}[$L^2$ transport inequalities] \label{thm_w2_stein_bnd}

For $k = 2$, we have 
$$W_2(\mu, \gamma) \leq \max (\bar{S}_2(\mu)(1 - \log(\bar{S}_2/\bar{S}_1)), \bar{S}_2(\mu)).$$

For $k \geq 3$, we have $W_2(\mu, \gamma) \leq 2k\bar{S}_1(\mu)^{1-1/(k-1)}\bar{S}_k(\mu)^{1/(k-1)}$.
\end{thm}

The first inequality will allow to improve the rate of convergence in the CLT in $W_2$ distance for measures having its moments of order 3 equal to zero. As we will later see, when $k \geq 3$, these inequalities are not satisfactory for applications to CLTs. 

\begin{proof}
As pointed out in \cite{OV00}, we have 
$$W_2(\mu, \gamma) \leq \int_0^{\infty}{I(\mu_s)^{1/2}ds}.$$

For $k = 2$, we have for any $t$ 
\begin{align*}
W_2(\mu, \gamma) &\leq \int_0^t{I(\mu_s)^{1/2}ds} + \int_t^{\infty}{I(\mu_s)^{1/2}ds} \\
&\leq \bar{S}_1(\mu)\int_0^t{\frac{e^{-2s}}{\sqrt{1 - e^{-2s}}}ds} + \bar{S}_2(\mu)\int_t^{\infty}{\frac{e^{-3s}}{1 - e^{-2s}}ds} \\
&\leq \sqrt{1 - e^{-2t}}\bar{S}_1(\mu) - \frac{1}{2}e^{-t}\log(1 - e^{-2t})\bar{S}_2(\mu) \\
&\leq \sqrt{1 - e^{-2t}}\bar{S}_1(\mu) - \frac{1}{2}\log(1 - e^{-2t})\bar{S}_2(\mu)
\end{align*}
Optimizing in $t$ then leads to choosing $t$ such that $\sqrt{1 - e^{-2t}} = \min(\bar{S}_2/\bar{S}_1, 1)$. If $\bar{S}_2 \leq \bar{S}_1$, we end up with the bound $W_2 \leq \bar{S}_1(\mu)(1 - \log(\bar{S}_2/\bar{S}_1))$, and this upper bound, is larger than $\bar{S}_2$. Otherwise, we bound it by $\bar{S}_2 \geq \bar{S}_1$, and the desired bound holds either way. 

For $k \geq 3$, we similarly have
\begin{align*}
W_2(\mu, \gamma) &\leq \int_0^t{I(\mu_s)^{1/2}ds} + \int_t^{\infty}{I(\mu_s)^{1/2}ds} \\
&\leq \bar{S}_1(\mu)\int_0^t{\frac{e^{-2s}}{\sqrt{1 - e^{-2s}}}ds} + \sqrt{k!}\bar{S}_k(\mu)\int_t^{\infty}{\frac{e^{-(k+1)s}}{(1 - e^{-2s})^{k/2}}ds} \\
&\leq \sqrt{1 - e^{-2t}}\bar{S}_1(\mu) + \sqrt{k!}\bar{S}_k(\mu)\frac{e^{-(k-1)t}}{k-2}\left(\frac{1}{1 - e^{-2t}}\right)^{(k-2)/2} \\
&\leq \sqrt{1 - e^{-2t}}\bar{S}_1(\mu) + \sqrt{k!}\bar{S}_k(\mu)\left(\frac{1}{1 - e^{-2t}}\right)^{(k-2)/2}
\end{align*}
and taking $\sqrt{1-e^{-2t}} =(\sqrt{k!}\bar{S}_k(\mu)/\bar{S}_1(\mu))^{1/(k-1)}$ yields the result. The inequality could be improved, at the cost of clarity, but as far as we can see the sharper inequality obtained by this method does not significantly improve the outcomes in the applications. 
\end{proof}

\section{Improved rates of convergence in the classical CLT}

We are interested in the rate of convergence of the law of (normalized) sums of iid random variables $n^{-1/2}\sum_{i=1}^{n} X_i$ to their Gaussian limit. It is known that the rate of convergence in Wasserstein distance $W_2$ is of order $\sqrt{n}$ in general, as soon as the fourth moment is finite \cite{Bon}. However, it is possible to do a Taylor expansion of the distance as $n$ goes to infinity, and see that under moment constraints, the asymptotic rate of decay may improve. More precisely, \cite{BCG13, BCG14} shows that in dimension one, if the first $k+1$ moments of the random variables match with those of the standard Gaussian, then the Wasserstein distance (and the stronger relative entropy and Fisher information) asymptotically decays like $n^{-k/2}$. Non-asymptotic rates in dimension one were obtained in \cite{GR97} using a variant of Stein's method, and strong entropic rates under a Poincar\'e inequality and after regularization by convolution with a Gaussian measure were obtained in \cite{Mic03}, still in dimension one. \cite{ABBN04} gives a sharp non-improved rate of convergence in the entropic CLT in dimension one in the classical case (i.e. without the extra moment constraints satisfied), without any regularization. See also \cite{BN12} for a multi-dimensional extension when the measure is additionally assumed to be log-concave. 

It is possible to use Stein's method to give simple proofs of this decay rate \cite{Ros11, LNP15}. In particular, \cite{CFP17} proves a monotone decay of the Stein discrepancy, which immediately implies the quantitative CLT as soon as the Stein discrepancy of a single variable is finite. 

We consider the usual setting for the classical CLT: a sequence $(X_i)$ of iid random variables with distribution $\mu$, and the normalized sum
$$U_n := \frac{1}{\sqrt{n}}\sum_{i=1}^n X_i$$
whose law we shall denote by $\mu_n$. 

The aim of this section is to show similar results for higher-order discrepancies. The starting point is the following construction of Stein kernels of the second type for sums of independent random variables, which is an immediate generalization of the same result for $k=1$.   

\begin{lem}
Let $\bar{\tau}_{k,1}$ be a $k$-th order Stein kernel for $\mu$. Then 
$$\bar{\tau}_{k,n} (m):= n^{-(k+1)/2}\mathbb{E}\left[\sum \bar{\tau}_{k,1}(X_i) | n^{-1/2}\sum X_i = m \right]$$
is a $k$-th order Stein kernel for $\mu_n$. 
\end{lem}

\begin{proof}
This can easily be checked by induction on $k$ via \eqref{eq_stein_ho_iter}. The case $k = 1$ is well-known \cite{LNP15}. 
\end{proof}

As a consequence, we obtain bounds on the rate of convergence of the Stein discrepancies: 

\begin{cor}
Assume that all the mixed moments of order less than $k+1$ of $X_1$ are the same as those of the standard Gaussian. Then 
$$\bar{S}_k(\mu_n)^2 \leq n^{-k}\bar{S}_k(\mu)^2.$$
\end{cor}

We then obtain a rate of convergence in the multivariate CLT for the Zolotarev distances $d_{Zol,k}$ as an immediate consequence of the comparison from Theorem \ref{thm_comp_zol}: 

\begin{cor}
Assume that $\mu$ satisfies a Poincar\'e inequality with constant $C_P$, and that all its mixed moments of order less than $k+1$ match with those of the standard Gaussian measure. Let $\mu_n$ be the law of $U_n$. Then
$$d_{Zol,k} (\mu_n, \gamma) \leq \frac{\sqrt{C_P^{k-1}(C_P-1)d}}{n^{-k/2}}.$$
\end{cor}

Such results have been in dimension one (and for random vectors with independent coordinates) in \cite{Gau16, Gau15}. See also \cite{GR97} for related results. 

Combined with the logarithmic Sobolev inequality and Lemma \ref{lem_fisher_decay}, this also yields a multi-dimensional extension of a result of \cite{Mic03} on improved entropic CLTs for regularized measures, with more explicit quantitative prefactors. 

In the case $k=3$, due to Theorem \ref{thm_w2_stein_bnd}, we can upgrade the distance to $W_2$, losing however a logarithmic factor: 

\begin{cor} \label{cor_w2_rate}
Assume that all the mixed moments of order less than three of $X_1$ are the same as those of the standard Gaussian, and that its law satisfies a Poincar\'e inequality with constant $C_P$. Then 
$$W_2(\mu_n, \gamma) \leq \frac{\sqrt{dC_P(C_P-1)}}{n}\left(1 + \frac{1}{2}\log n + \frac{1}{2}\log C_P\right)$$
as soon as $n \geq \sqrt{dC_P(C_P-1)}$. If additionally the mixed fourth moments match with those of the Gaussian, we get
$$W_2(\mu_n, \gamma) \leq \frac{2\sqrt{dC_P(C_P-1)}}{(k-1)n}.$$
\end{cor}

\begin{proof}
The first inequality is obtained by plugging the upper bounds on discrepancies in the bounds of Theorem \ref{thm_w2_stein_bnd}, while using the fact that $x \rightarrow x(1-\log x)$ is increasing on $[0, 1]$. The second inequality is obtained by using the 2nd order kernels, and with our estimates using even higher order kernels does not improve the bounds. 
\end{proof}

When $k=2$, we only miss the sharp asymptotic rate of \cite{BCG13} by a logarithmic factor. However, under higher moment constraints we know that the asymptotic rate is much better than $n^{-1}$ (at least in dimension one), so this result is not satisfactory. 

For the entropy without regularization, we obtain the following rates under the assumption that mixed third moments are equal to zero:  

\begin{prop}
Assume that the law of the $X_i$ satisfies a Poincar\'e inequality and that the moments of order less than three agree with those of the standard Gaussian measure. Then
$$\Ent_{\gamma}(\mu_n) \leq \frac{2C_P \sqrt{d}}{n}I(\mu)^{1/2}.$$
\end{prop}

This eliminates a logarithmic factor from previous results of \cite{LNP15} in this particular case, but once again does not give the expected sharp decay rate under the moment assumptions.

\begin{proof}
This estimate is obtained by applying the HSI inequality with $k=2$ and the fact that Fisher information is monotone along the CLT \cite{ABBN04a}. 
\end{proof}

\vspace{5mm}

\underline{\textbf{Acknowledgments}}: This work was supported by the Projects MESA (ANR-18-CE40-006) and EFI (ANR-17-CE40-0030) of the French National Research Agency (ANR), ANR-11-LABX-0040-CIMI within the program ANR-11-IDEX-0002-02 and the France-Berkeley Fund. I would also like to thank Guillaume C\'ebron, Thomas Courtade, Michel Ledoux and G\'esine Reinert for discussions on this topic.

\end{document}